\documentclass[12pt]{amsart}

\usepackage{ljm-auth}

\usepackage{graphicx}

\newtheorem{definition}{Definition}
\newtheorem{theorem}{Theorem}

\newtheorem{lemma}{Lemma}

\author{S.~V. Gryshchuk$^{\ast}$ and S.~A. Plaksa$^{\ast\ast}$
} \crauthor{S.~V. GRYSHCHUK AND   S.~A. PLAKSA}

\tit{A Schwartz-type boundary value problem in a biharmonic plane}
\shorttit{A Schwartz-type BVP in a biharmonic plane}

\begin{document}

\maketit

\address{Institute of Mathematics, National Academy of
Sciences of Ukraine, Tereshchenkivska Str. 3, 01601, Kiev-4,
Ukraine}

\email{$^{\ast}$ serhii.gryshchuk@gmail.com, $^{\ast\ast}$
plaksa62@gmail.com}

\vskip 2mm



\abstract{A commutative algebra $\mathbb{B}$ over the field of
complex numbers with the bases $\{e_1,e_2\}$ satisfying the
conditions $(e_1^2+e_2^2)^2=0$, $e_1^2+e_2^2\ne 0$, is considered.
The algebra $\mathbb{B}$ is associated with the biharmonic
equation. Consider a Schwartz-type boundary value problem on
finding a monogenic function of the type
$\Phi(xe_1+ye_2)=U_{1}(x,y)\,e_1+U_{2}(x,y)\,ie_1+
U_{3}(x,y)\,e_2+U_{4}(x,y)\,ie_2$, $(x,y)\in D$, when values of
two components $U_1$, $U_4$ are given on the boundary of a domain
$D$ lying in the Cartesian plane $xOy$. We develop a method of its
solving which is based on expressions of monogenic functions via
corresponding analytic functions of the complex variable. For a
half-plane and for a disk, solutions are obtained in explicit
forms by means of Schwartz-type integrals.}


\notes{0}{
 \thank{This research is partially supported by Grant
of Ministry of Education and Science of Ukraine (Project No.
 0116U001528)}}

\section{Monogenic functions in a biharmonic algebra}
\label{intr}
\begin{definition}\label{D:bih-alg}
 An associative commutative two-dimensional algebra $\mathbb B$ with
the unit $1$ over the field of complex numbers $\mathbb C$ is
called {\it biharmonic} \em (\em see \cite{KM-BFf,M-BB}\em) \em if
in $\mathbb B$ there exists a basis $\{e_1,e_2\}$ satisfying the
conditions
\[ (e_1^2+e_2^2)^2=0,\qquad e_1^2+e_2^2\ne 0\,.\]

Such a basis $\{e_1,e_2\}$ is also called {\it biharmonic}.
\end{definition}

In the paper \cite{M-BB} I.~P. Mel'nichenko proved that there
exists the unique biharmonic algebra $\mathbb{B}$, and he
constructed all biharmonic bases in $\mathbb{B}$. Note that the
algebra $\mathbb{B}$ is isomorphic to four-dimensional over the
field of real numbers $\mathbb R$ algebras considered by
A.~Douglis \cite{Douglis-53} and L.~Sobrero~\cite{Sodbero}.

In what follows, we consider a biharmonic basis $\{e_1,e_2\}$ with
the following multiplication table (see \cite{KM-BFf}):
\begin{equation} \label{tab_umn_bb}
e_1=1,\qquad e_2^2=e_1+2ie_2\,,
\end{equation}
where $i$\,\, is the imaginary complex unit.
 We consider also a basis $\{1,\rho\}$ (see \cite{M-BB}), where a
nilpotent element
\begin{equation} \label{rho}
\rho=2e_1+2ie_2
\end{equation}
satisfies the equality\, $\rho^{2}=0$\,.

We use the euclidian norm $\|a\|:=\sqrt{|z_1|^2+|z_2|^2}$ in the
algebra $\mathbb{B}$, where\, $a=z_1e_1+z_2e_2$ and $z_1, z_2\in
\mathbb{C}$.

Consider a {\it biharmonic plane} $\mu:=\{\zeta=x\,e_1+y\,e_2 :
x,y\in\mathbb R\}$ which is a linear span of the elements
$e_1,e_2$ of the biharmonic basis (\ref{tab_umn_bb}) over the
field of real numbers $\mathbb R$. With a domain $D$ of the
Cartesian plane $xOy$ we associate the congruent domain
$D_{\zeta}:= \{\zeta=xe_1+ye_2 : (x,y)\in D\}$ in the biharmonic
plane $\mu$ and the congruent domain $D_{z}:= \{z=x+iy : (x,y)\in
D\}$ in the complex plane $\mathbb{C}$. Its boundaries are denoted
by $\partial D$, $\partial D_{\zeta}$ and $\partial D_z$,
respectively. Let $\overline{D_{\zeta}}$ (or $\overline{D_{z}}$)
be the closure of domain $D_{\zeta}$ (or $D_{z}$). In what
follows, $\zeta=x\,e_1+y\,e_2$,\,\, $z=x+iy$\,\, and\,
$x,y\in\mathbb{R}$.

Any function $\Phi\colon D_{\zeta}\longrightarrow \mathbb{B}$ has
an expansion of the type
\begin{equation}\label{mon-funk}
\Phi(\zeta)=
U_{1}(x,y)\,e_1+U_{2}(x,y)\,ie_1+U_{3}(x,y)\,e_2+U_{4}(x,y)\,ie_2\,,
\end{equation}
where $U_{l}\colon D\longrightarrow \mathbb{R}$, $l=1,2,3,4$, are
real-valued component-functions.
 We shall use the following notation:
$\mathrm{U}_{l}\left[\Phi\right]:=U_{l}$, $l=1,2,3,4$.

\begin{definition}\label{D:monogenic}
A function $\Phi \colon D_{\zeta} \longrightarrow \mathbb{B}$ is
{\em monogenic} in a domain $D_{\zeta}$ if it has the classical
derivative $\Phi'(\zeta)$ at every point $\zeta\in
D_{\zeta}$\emph{:}
\[\Phi'(\zeta):=\lim\limits_{h\to 0,\, h\in \mu}
\bigl(\Phi(\zeta+h)-\Phi(\zeta)\bigr)\,h^{-1}\,.\]
\end{definition}

It is proved in \cite{KM-BFf} that a function $\Phi\colon
D_{\zeta}\longrightarrow \mathbb{B}$ is monogenic in $D_{\zeta}$
if and only if its each real-valued component-function in
(\ref{mon-funk}) is real differentiable in $D$ and the following
analog of the Cauchy~-- Riemann condition is fulfilled:
\[\frac{\partial \Phi(\zeta)}{\partial y}\,e_1=\frac{\partial
\Phi(\zeta)}{\partial x}\,e_2\quad \forall\, \zeta \in
D_{\zeta}\,.\]

All component-functions $U_{l}$, $l=1,2,3,4$, in the expansion
(\ref{mon-funk}) of any monogenic function  $\Phi \colon
D_{\zeta}\longrightarrow \mathbb{B}$ are biharmonic functions
(cf., e.g., \cite{GrPl_umz-09,Conrem_13}), i.e., satisfy the
biharmonic equation in $D$:
    \[ \Delta^{2}U(x,y)\equiv \frac{\partial^{4}U(x,y)}{\partial x^4}+
   2\frac{\partial^{4}U(x,y)}{\partial x^2\partial y^2}+
   \frac{\partial^4 U(x,y)}{\partial y^4} =0.\]

Every monogenic function  $\Phi \colon D_{\zeta}\longrightarrow
\mathbb{B}$ is expressed via two corresponding  analytic functions
$F \colon D_{z}\longrightarrow\mathbb{C}$, $F_{0}\colon
D_{z}\longrightarrow\mathbb{C}$ of the complex variable\, $z$\,
in the form (cf., e.g., \cite{GrPl_umz-09,Conrem_13}):
\begin{equation}\label{kov}
\Phi(\zeta)=F(z)e_1-\left(\frac{iy}{2}\,F'(z)-F_0(z)\right)\rho\
\quad \forall\, \zeta \in D_{\zeta}.
\end{equation}
The equality  (\ref{kov}) establishes  one-to-one correspondence
between monogenic functions  $\Phi$ in the domain  $D_{\zeta}$ and
pairs of complex-valued analytic functions $F,F_{0}$ in the domain
$D_{z}$.

Using the equality  (\ref{rho}), we rewrite the expansion
(\ref{kov}) for all $\zeta\in D_{\zeta}$ in the basis
$\{e_1,e_2\}$:
\begin{equation}\label{kov-e1-e2}
\Phi(\zeta)=\Bigl(F(z)-iyF'(z)+2F_{0}(z)\Bigr)e_1
+i\Bigl(2F_{0}(z)-iyF'(z)\Bigr)e_2\,.
\end{equation}

\section{($k$-$m$)-problem for monogenic functions}
\label{sec:(k-m)-pr}  V.\,F. Kovalev \cite{Kovalov} considered the
following boundary value problem: to find a continuous function
$\Phi : \overline{D_{\zeta}} \longrightarrow\mathbb{B}$ which is
monogenic in a domain $D_{\zeta}$ when values of two
component-functions in (\ref{mon-funk}) are given on the boundary
$\partial D_{\zeta }$, i.e., the following boundary conditions are
satisfied:
\[U_{k}(x,y)=u_{k}(\zeta)\,,\quad
U_{m}(x,y)=u_{m}(\zeta)\qquad\forall\, \zeta \in
\partial D_{\zeta}\]
for $1\le k < m\le 4$\,, where $u_{k}$ and $u_{m}$ are given
functions.

We assume additionally that the sought-for function $\Phi$ has the
limit
\[\lim\limits_{\|\zeta\|\to \infty,\,\zeta\in
D_{\zeta}}\Phi(\zeta)=:\Phi(\infty)\in \mathbb{B}\]
 in the case
where the domain $D_{\zeta}$ is unbounded as well as every given
function $u_l$,\,\, $l\in\{k,m\} $, has a finite limit
\begin{equation} \label{lim_infty}
u_l(\infty):=\lim\limits_{\|\zeta\|\to \infty,\, \zeta\in \partial
D_{\zeta}} u_l(\zeta)
\end{equation}
if $\partial D_{\zeta}$ is unbounded.

We shall call such a problem by the ($k$-$m$)-problem. V.\,F.
Kovalev \cite{Kovalov} called it by a {\it biharmonic Schwartz
problem} owing to its analogy with the classical Schwartz problem
on finding an analytic function of the complex variable when
values of its real part are given on the boundary of domain.

It was established in \cite{Kovalov} that all biharmonic Schwartz
problems are reduced to the main three problems: the (1-2)-problem
or the (1-3)-problem or the (1-4)-problem.

It is shown (see \cite{Kovalov,IJPAM_13,mon-f-bih-BVP-MMAS}) that
the main biharmonic problem is reduced to the (1-3)-problem. In
\cite{IJPAM_13}, we investigated the (1-3)-problem for cases where
$D_{\zeta}$ is either a half-plane or a unit disk in the
biharmonic plane. Its solutions were found in explicit forms with
using of some integrals analogous to the classic Schwarz integral.
In \cite{mon-f-bih-BVP-MMAS}, using a hypercomplex analog of the
Cauchy type integral, we reduced the (1-3)-problem to a system of
integral equations and established sufficient conditions under
which this system has the Fredholm property. It was made for the
case where the boundary of domain belongs to a class being wider
than the class of Lyapunov curves that was usually required in the
plane elasticity theory (cf., e.g.,
\cite{Mikhlin,Mush_upr,Lurie_engl,Kakh,Bogan}).

In \cite{1-4-ArxiV}, there is considered a relation between
(1-4)-problem and boundary value problems of the plane elasticity
theory. Namely, there is considered a problem on finding an
elastic equilibrium for isotropic body occupying $D$ with given
limiting values of partial derivatives $\frac{\partial u}{\partial
x}$\,, $\frac{\partial v}{\partial y}$ for displacements
$u=u(x,y)$\,, $v=v(x,y)$ on the boundary $\partial D$. In
particular, it is shown in \cite{1-4-ArxiV} that such a problem is
reduced to (1-4)-problem.

In this paper we develop methods for solving the (1-4)-problem.
Obtained results are mostly analogous to appropriate results in
\cite{IJPAM_13,mon-f-bih-BVP-MMAS} dealing with the (1-3)-problem,
but in contradistinction to the (1-3)-problem, which is solvable
in a general case if and only if a certain natural condition is
satisfied, the (1-4)-problem is solvable unconditionally.

Note that hypercomplex methods and suitable "analytic" functions
are used for investigation of elliptic partial differential
equations in the papers
\cite{Douglis-53,Sodbero,Diaz-46,Gilb-Hile-74,
Gilb-Wendl-Proc-75,Edenhofer,Hile-79,Sold-ell-vys-por,Sold-Izv-06,Boj-zad-Neymn-Big-Equ}.

 \section{Solving process of (1-4)-problem
 by means analytic functions of the complex
 variable}
 \label{sec:(1-4)-an-f-C}
Consider a method for solving the (1-4)-problem that is based on
the expression (\ref{kov}) of monogenic function by means of
appropriate analytic functions of the complex variable.

For a continuous function $u : \partial D_{\zeta} \longrightarrow
\mathbb{R}$, by $\widehat{u}$ we denote the function defined on
$\partial D_z$ by the equality $\widehat{u}(z)=u(\zeta)$ for all
$z \in \partial D_z$.

In what follows, we assume that the domain $D_{z}$ is simply
connected (bounded or unbounded), and in this case we shall say
that the domain $D_{\zeta}$ is also simply connected.

\begin{lemma}\label{L:pr-1-4-Lem1}
Let $ u_l \colon \partial D_{\zeta} \longrightarrow \mathbb{R}$,
$l \in \{1,4\}$,  be continuous functions and  $\Phi \colon
\overline{D_{\zeta}} \longrightarrow\mathbb{B}$ be a solution of
the \em (1-4)\em-problem. Then the function  $F$ in \eqref{kov} is
a solution of the Schwartz problem on finding a continuous
function  $F \colon  \overline{D_z} \longrightarrow\mathbb{C}$
which is analytic in $D_z$ and satisfies the boundary
condition\emph{:}
\begin{equation}\label{z-Sh-F}
\mathrm{Re}\,F(t)=\widehat{u}_{1}(t)-\widehat{u}_{4}(t) \quad
\forall\, t \in\partial D_z.
\end{equation}
\end{lemma}
\begin{proof}
Consider the linear continuous multiplicative functional $f \colon
{\mathbb B}\longrightarrow{\mathbb C}$  such that $f(\rho)=0$ and
$f(1)=1$. Then, it follows from the equality (\ref{rho}) that
$f(e_2)=i$.

Thus,  acting by the functional  $f$ on the equalities
(\ref{mon-funk}), (\ref{kov}), we obtain the relations
\[F(z)=f\Bigl(\Phi(\zeta)\Bigr)=U_1(x,y)-U_4(x,y)+i\Bigl(U_2(x,y)+U_3(x,y)\Bigr)
\quad\forall\,z\in D_z\,.\]

Inasmuch as the functional $f$ is continuous and $\Phi$ is a
solution of the (1-4)-problem, the analytic function $F$ permits a
continuous extendibility to the boundary $\partial D_z$ and
satisfies the boundary condition (\ref{z-Sh-F}).
\end{proof}

\begin{lemma}\label{L:pr-1-4-Lem2}
 Let the conditions of Lemma \em
\ref{L:pr-1-4-Lem1} \em be satisfied and, furthermore, the
function $yF'(z)$ permit a continuous extendibility to the
boundary $\partial D_z$,  where $F$ is a solution of the Schwartz
problem with boundary condition  \eqref{z-Sh-F}. Then the function
$F_0$ in \eqref{kov} is a solution of the Schwartz problem on
finding a continuous function  $F_0 \colon \overline{D_z}
\longrightarrow\mathbb{C}$, which is analytic in $D_z$ and
satisfies the boundary condition
\begin{equation}\label{z-Sh-F_0}
\mathrm{Re}\,F_0(t)=\frac{1}{2}\Bigl(\widehat{u}_{4}(t)-\mathrm{Im}\,t\,\,\mathrm{Im}\,
\lim\limits_{z\to t,\, z\in D_z} F'(z)\Bigr) \quad \forall\, t
\in\partial D_z.
\end{equation}
\end{lemma}
\begin{proof}
We obtain the equality
\[U_4[\Phi(\zeta)]=
2\,\mathrm{Re}\,F_0(z)+y\,\mathrm{Im}\,F'(z)\quad \forall\,
\zeta\in D_{\zeta}\]
 as a corollary of the equality (\ref{kov-e1-e2}).

Under the conditions of Lemma~\ref{L:pr-1-4-Lem1}, the Schwartz
problem with boundary condition \eqref{z-Sh-F} is solvable. Taking
into account that a function  $yF'(z)$ permits a continuous
extendibility to the boundary $\partial D_z$, we conclude that the
function $F_0$ is a solution of the Schwartz problem with boundary
condition (\ref{z-Sh-F_0}).
\end{proof}

By virtue of Lemmas \ref{L:pr-1-4-Lem1} and \ref{L:pr-1-4-Lem2},
we can assert that a solving of the (1-4)-problem is reduced to
successive solving processes of two Schwartz problems with
boundary conditions (\ref{z-Sh-F}), (\ref{z-Sh-F_0}),
respectively. Therefore, we get the following theorem.

\begin{theorem}\label{Th:resh-(1-4)-an-f}
Let $ u_l \colon \partial D_{\zeta} \longrightarrow \mathbb{R}$,
$l \in \{1,4\}$, be continuous functions and, furthermore, the
function $yF'(z)$ permit a continuous extendibility to the
boundary $\partial D_z$, where  $F$  is a solution of the Schwartz
problem with boundary condition \eqref{z-Sh-F}. Then a solution of
the \em (1-4)\em -problem is expressed by the formula \eqref{kov}
or, the same, by the formula \eqref{kov-e1-e2}, where the function
$F_0$ is a solution of the Schwartz problem with boundary
condition \eqref{z-Sh-F_0}. \end{theorem}


A particular case of Theorem~\ref{Th:resh-(1-4)-an-f} is the
following theorem, where all solutions of the homogeneous
(1-4)-problem, i.e., with $u_1=u_4 \equiv 0$, are described  for an
arbitrary (bounded or unbounded) simply connected domain
$D_{\zeta}$.

 \begin{theorem}\label{Th:odnor-(1-4)}
The general solution of the homogeneous  \emph{(1-4)}-problem for
an arbitrary simply connected domain $D_{\zeta}$ is expressed by
the formula
\begin{equation}\label{form-odn-(1-4)}
\Phi(\zeta)=a_1 ie_1+a_2 e_2,
\end{equation}
 where $a_1$, $a_2$ are any real constants.
 \end{theorem}

\begin{proof} By Theorem \ref{Th:resh-(1-4)-an-f}, a solving process of
the homogeneous (1-4)-problem consists of consecutive finding of
solutions of two homogeneous Schwartz problems, viz.:

a) to find an analytic in $D_z$ function $F$ satisfying the
boundary condition $\mathrm{Re}\,F(t)= 0$ for all $t\in\partial
D_{z}$. As a result, we have  $F(z)=ai$, where $a$ is an arbitrary
real constant;

b) to find similarly an analytic in $D_z$ function $F_0$
satisfying the boundary condition $\mathrm{Re}\,F_0(t)= 0$ for all
$t\in\partial D_{z}$.

Consequently, getting a general solution of the homogeneous
(1-4)-problem in the form (\ref{kov-e1-e2}), we can rewrite it in
the form (\ref{form-odn-(1-4)}).
\end{proof}

Having an intention to develop a method for solving the
inhomogeneous (1-4)-problem without an essential in Theorem
\ref{Th:resh-(1-4)-an-f} assumption that the function  $yF'(z)$
permits a continuous extendibility to the boundary  $\partial
D_z$, we consider primarily questions about finding solutions of
the inhomogeneous (1-4)-problem for some canonical domains,
namely: a half-plane and a disk.

\section{(1-4)-problem for a half-plane.}\label{(1-4)-half-plane}
Consider the (1-4)-problem in the case where the domain
$D_{\zeta}$ is the half-plane $\Pi^+:=\{\zeta=xe_1+ye_2 : y>0\}$.

Our aim is to find an explicit formula of solution of the
(1-4)-problem for the half-plane $\Pi^+$ under the assumption that
for every given function $u_l \colon \mathbb{R}\longrightarrow
\mathbb{R}$, $l\in\{1,4\}$, its modulus of continuity
\[\omega_{\mathbb R}(u_l, \varepsilon)=\sup\limits_{\tau_1, \tau_2 \in \mathbb{R} : |\tau_1-\tau_2|\le \varepsilon }
\left|u_l(\tau_1)-u_l(\tau_2)\right| \]
 and the local centered
(with respect to the infinitely remote point) modulus of
continuity
\[\omega_{\mathbb R, \infty}(u_l, \varepsilon)=\sup\limits_{\tau \in \mathbb{R} : |\tau|\ge 1/\varepsilon }
\left|u_l(\tau)-u_l(\infty)\right|\]
 satisfy the Dini conditions
\begin{equation} \label{usl_Dini}
\int\limits_{0}^{1}\frac{\omega_{\mathbb R}(u_l,
\eta)}{\eta}\,d\,\eta<\infty,
\end{equation}
 \begin{equation} \label{usl_Dini+infty}
\int\limits_{0}^{1}\frac{\omega_{\mathbb R,  \infty}(u_l,
\eta)}{\eta}\,d\,\eta<\infty,
\end{equation}

In the following theorem all integrals along the real axis are
understood in the sense of their Cauchy principal values.

\begin{theorem} \label{Th:Z_1-4}
 Let every function  $u_l \colon
\mathbb{R}\longrightarrow \mathbb{R}$, $l\in\{1,4\}$, have a
finite limit of the type \eqref{lim_infty} and the conditions
\eqref{usl_Dini}, \eqref{usl_Dini+infty} be satisfied. Then the
general solution of the {\emph{(1-4)}}-problem for the half-plane
$\Pi^{+}$ is expressed by the formula
 \begin{equation} \label{sol_1-4-pivpl}
 \Phi(\zeta)=S_{\Pi^+}[u_1](\zeta)\,e_1+S_{\Pi^+}[u_4](\zeta)\,ie_2 +a_1 ie_1+a_2
 e_2\,,
 \end{equation}
where
 \[ S_{\Pi^+}[u_l](\zeta):= \frac{1}{\pi
i}\int\limits_{-\infty}^{+\infty}\frac{u_l(t)(1+t\zeta)}{(t^2+1)}(t-\zeta)^{-1}\,dt\quad
\forall\, \zeta\in \Pi^{+}\,,\quad l\in\{1,4\},\]
 and \, $a_1$, $a_2$ are any real constants.
\end{theorem}

\begin{proof}
For each function  $u_l$, $l\in\{1,4\}$, the following equalities
are fulfilled
\[\lim\limits_{\zeta\to \xi,\,
\zeta\in\Pi^+}S_{\Pi^+}[u_l](\zeta)=u_l(\xi)+\frac{1}{\pi
i}\int\limits_{-\infty}^{\infty}\frac{u_l(t)}{t^2+1}\frac{1+t\xi}{t-\xi}\,dt
\quad \forall\,\xi \in \mathbb{R}\,,\]
\[\lim\limits_{\|\zeta\|\to \infty,\,
\zeta\in\Pi^+}S_{\Pi^+}[u_l](\zeta)=u_l(\infty)-\frac{1}{\pi
i}\int\limits_{-\infty}^{\infty}u_l(t)\frac{t}{t^2+1}\,dt\,,\]
which are proved in Theorem~1 of the paper  \cite{IJPAM_13}. It
follows from these equalities that the function
 \begin{equation} \label{sol_1-ch-1-4}
\Phi(\zeta)=S_{\Pi^+}[u_1](\zeta)\,e_1+S_{\Pi^+}[u_4](\zeta)\,ie_2
\end{equation}
is a solution of the  (1-4)-problem for the half-plane $\Pi^+$.

The general solution of the (1-4)-problem in the form
(\ref{sol_1-4-pivpl}) is obtained by summarizing the particular
solution (\ref{sol_1-ch-1-4}) of the inhomogeneous (1-4)-problem
and the general solution  (\ref{form-odn-(1-4)}) of the
homogeneous (1-4)-problem.
\end{proof}

Note that in \cite{Kovalov} the (1-4)-problem for the half-plane
is solved under complementary assumptions, viz.: the function
$u_l$ belongs to a H\"{o}lder space and $u_l(\infty)=0$ for
$l\in\{1,4\}$, that imply the conditions \eqref{usl_Dini},
\eqref{usl_Dini+infty}.

\section{(1-4)-problem for a disk}\label{(1-4)-disk}
Now, let $D_{\zeta}:=\{\zeta=xe_1+ye_2 : \|\zeta\|\le 1\}$ be the
unit disk in the biharmonic plane $\mu$ and $D_z:=\{z=x+iy : |z|\le
1\}$ be the unit disk in the complex plane $\mathbb{C}$.

In order to construct a solution of the (1-4)-problem for the disk
$D_{\zeta}$, as well as in the case of (1-3)-problem in the paper
\cite{IJPAM_13}, we use the integral
\begin{equation}\label{big_int_Sw_c}
S_{D_{\zeta}}[u](\zeta):= \frac{1}{2\pi i}\int\limits_{\partial
D_{\zeta}}
u(\tau)(\tau+\zeta)(\tau-\zeta)^{-1}\,\tau^{-1}\,d\tau\qquad
\forall\, \zeta\in D_{\zeta}
\end{equation}
being an analog of the complex Schwartz-type integral.

It is proved in \cite{IJPAM_13} that if the modulus of continuity
\[\omega(u,\varepsilon):=\sup\limits_{\zeta_1, \zeta_2 \in\partial
D_{\zeta} : \|\zeta_1-\zeta_2\|\le \varepsilon}
\|u(\zeta_1)-u(\zeta_2)\| \]
 of the function $u : \partial
D_{\zeta} \longrightarrow \mathbb{R}$ satisfies the Dini condition
\begin{equation} \label{usl_Dini_c}
\int\limits_{0}^{1}\frac{\omega(u, \eta)}{\eta}\,d\,\eta
<\infty\,,
\end{equation}
then the integral (\ref{big_int_Sw_c}) has limiting values on
$\partial D_{\zeta}$. Here we rewrite a formula for the mentioned
limiting values (cf. the formulas (25), (26) in \cite{IJPAM_13})
in the following form
\[ \lim\limits_{\xi\to\zeta,\, \xi\in D_{\zeta}} S_{D_{\zeta}}[u](\xi)=u(\zeta)\,e_1+S_{0}[\widehat{u}](z)e_1+
\]
\begin{equation} \label{gr_big_Sw_c}
+\biggl(\frac{x-iy}{2\pi}\int\limits_{\partial D_z}
\frac{\widehat{u}(t)}{t^2}\,dt+
\frac{1}{2\pi}\int\limits_{\partial D_z}
\frac{\widehat{u}(t)}{t^3}\,dt\biggr)(e_2-ie_1)
\quad\forall\,\zeta \in \partial D_{\zeta}\,,
\end{equation}
where
 \[S_0[\widehat{u}](z):=\frac{1}{2\pi i}\,\lim_{\varepsilon\to 0+0} \int\limits_{\{t\in\partial
D_z\, : \, |t-z|\ge\varepsilon\}
}\widehat{u}(t)\,\frac{t+z}{t-z}\,\frac{dt}{t}\,, \qquad
z\in\partial D_{z}\,,\]

An explicit formula for solution of the (1-4)-problem for the unit
disk is obtained in the following theorem.

\begin{theorem} \label{L:pr-1-4-M}
Let  functions $ u_l \colon \partial D_{\zeta} \longrightarrow
\mathbb{R}$, $l\in \{1,4\}$, satisfy conditions of the type
\eqref{usl_Dini_c}. Then the general solution of \em (1-4)\em
-problem is expressed in the form
\[\Phi(\zeta)=S_{D_{\zeta}}[u_1](\zeta)\,e_1+S_{D_{\zeta}}[u_4](\zeta)\,ie_2+\]
\begin{equation} \label{sol_1-4_d} +
\Bigl((b_1+ib_2)\zeta+b\Bigr)(e_1+ie_2)+ a_1\,ie_1+a_2\,e_2\,,
 \end{equation}
where
 \[b_1:=-\frac{1}{2\pi}\,\mathrm{Im}\,\int\limits_{\partial
D_z}\frac{\widehat{u}_{1}(t)-\widehat{u}_{4}(t)}{t^2}\,dt\,,\qquad
b_2:=-\frac{1}{2\pi}\,\mathrm{Re}\,\int\limits_{\partial
D_z}\frac{\widehat{u}_{1}(t)-\widehat{u}_{4}(t)}{t^2}\,dt\,,\]
\[b:=-\frac{1}{2\pi}\,\mathrm{Im}\,\int\limits_{\partial
D_z}\frac{\widehat{u}_{1}(t)-\widehat{u}_{4}(t)}{t^3}\,dt\,,\]
 and $a_1, a_2$ are any real constants.
\end{theorem}

\begin{proof}
Let us prove that there exists a particular solution of the
(1-4)-problem in the form
  \[\Phi(\zeta)=
S_{D_{\zeta}}[u_1](\zeta)\,e_1+S_{D_{\zeta}}[u_4](\zeta)\,ie_2+\]
\begin{equation} \label{chast-sol_1-4_d}
+\Bigl(b_1e_1+b_2ie_1+b_3e_2+b_4ie_2\Bigr)\zeta+
c_1\,e_1+c_2\,ie_2\,,
 \end{equation}
where unknown coefficients  $b_1, b_2, b_3, b_4, c_1, c_2$ are need
to be found.

In order to single out components
 $U_l[\Phi]$, $l\in\{1,4\}$, of limiting values of the function (\ref{chast-sol_1-4_d}) on $\partial D_{z}${\bf,}
we use the equality (\ref{gr_big_Sw_c}) and get
 \[U_1[\Phi(\zeta)]=u_1(\zeta)+(B_1-B_4+b_1)x-(A_1-A_4-b_3)y+D_1-D_4+c_1\,, \]
 \[U_4[\Phi(\zeta)]=u_4(\zeta)+(B_1-B_4+b_4)x-(A_1-A_4-b_2-2b_3)y+D_1-D_4+c_2 \]
for $\zeta\in\partial D_{z}$, where
 \[A_l:=\frac{1}{2\pi}\,\mathrm{Re}\int\limits_{\partial
D_z}\frac{\widehat{u}_{l}(t)}{t^2}\,dt,\quad
B_{l}:=\frac{1}{2\pi}\,\mathrm{Im}\int\limits_{\partial
D_z}\frac{\widehat{u}_{l}(t)}{t^2}\,dt,\]
\[D_{l}:=\frac{1}{2\pi}\,\mathrm{Im}\int\limits_{\partial
D_z}\frac{\widehat{u}_{l}(t)}{t^3}\,dt, \qquad l\in\{1,4\}.\]

Now, it is clear that the identities $U_l[\Phi(\zeta)]\equiv
u_l(\zeta)$, $l\in\{1,4\}$, hold on $\partial D_{z}$ if
$b_4=b_1=-(B_1-B_4)$, $b_3=-b_2=A_1-A_4$, $c_1=c_2=-(D_1-D_4)$.

Finally, substituting the found values for the coefficients $b_1$,
$b_2$, $b_3$, $b_4$, $c_1$, $c_2$ to the partial solution
(\ref{chast-sol_1-4_d}) of the inhomogeneous  (1-4)-problem and
adding the general solution (\ref{form-odn-(1-4)}) of the
homogeneous (1-4)-problem, after evident identical transformations
we obtain the formula (\ref{sol_1-4_d}) for the general solution
of the (1-4)-problem for the unit disk.
\end{proof}



\begin{thebibliography}{20}

\bibitem{KM-BFf}
 V.\,F.~Kovalev, I.\,P.~Mel'nichenko, 
Reports Acad. Sci. USSR, ser.~A. No.~8, 25--27 (1981) [in
Russian].

\bibitem{M-BB}
I.\,P.~Mel'nichenko, 
Ukr. Mat. Zh. {\bf 38} (2), 224--226 (1986) [in Russian];
  English transl. (Springer) in Ukr. Math. J.
 {\bf 38} (2), 252--254 (1986).

 \bibitem{Douglis-53} A.~Douglis,    
Communications on Pure and Applied Mathematics {\bf 6} (2),
 259--289 (1953).

\bibitem{Sodbero}
L.~Sobrero,  
{Ricerche di Ingegneria}  {\bf 13} (2),
 255--264 (1934). 

\bibitem{GrPl_umz-09}
 S.\,V.~Grishchuk, S.\,A.~Plaksa,  
{ Ukr. Mat. Zh.}  {\bf 61} (12), 1587--1596 [in Russian]; English
transl. (Springer) in  Ukr. Math.~J. {\bf 61} (12), 1865--1876
(2009).

\bibitem{Conrem_13}
 S.\,V.~Gryshchuk, S.\,A.~Plaksa,  "Basic Properties of Monogenic Functions
in a Biharmonic Plane", in  {\it  "Complex Analysis and Dynamical
Systems V", Contemporary Mathematics} {\bf 591} (Amer. Math. Soc.:
Providence, RI, 2013), pp.~127--134.

\bibitem{Kovalov}
 V.\,F.~Kovalev,   "Biharmonic Schwarz problem", {\it Preprint
No.~86.16.~ Institute of Mathematics, Acad. Sci. USSR}, (Inst. of
Math. Publ. House: Kiev, 1986)[in Russian].

\bibitem{IJPAM_13}
S.\,V.~Gryshchuk,  S.\,A.~Plaksa, 
International Journal of Pure and Applied Mathematics {\bf 83}
(1), 193--211 (2013) (on-line version:
http://www.ijpam.eu/contents/2013-83-1/13/13.pdf).

\bibitem{mon-f-bih-BVP-MMAS}
 S.\,V.~Gryshchuk,  S.\,A.~Plaksa, 
 Mathematical Methods in the Applied Sciences {\bf 39} (11), 2939--2952 (2016)
  / (wileyonlinelibrary.com)
DOI: 10.1002/mma.3741 /.

\bibitem{Mikhlin}
S.\,G.~Mikhlin, "The plane problem of the theory of elasticity",
in {\it Trans. Inst. of seismology, Acad. Sci. USSR. no.~65}
(Acad. Sci. USSR Publ. House: Moscow--Leningrad, 1935) [in
Russian].

\bibitem{Mush_upr}
 N.\,I.~Muskhelishvili, {\it  Some basic problems of the mathematical theory of elasticity.
Fundamental equations, plane theory of elasticity, torsion and
bending. English transl. from the 4th Russian edition by R.M.~Radok}
(Noordhoff International Publishing, Leiden, 1977).

\bibitem{Lurie_engl}
A.\,I.~Lurie, {\it  Theory of Elasticity. Engl. transl. by
A.~Belyaev} (Springer-Verlag, Berlin etc., 2005).

\bibitem{Kakh}N.\,S.~Kahniashvili, "Research of the plain problems
of the theory of elasticity by the method of the theory of
potentials", in  {\it Trudy Tbil. Univer. {\bf 50}} (Publ. by
Tbil. Univer.: Tbilisi, 1953) [in Russian].

\bibitem{Bogan}
Yu.\,A.~Bogan,  
{ Sib. Zh. Vychisl. Mat.}   
 {\bf 4} (1), 21~--~30 (2001) [in Russian].

 \bibitem{1-4-ArxiV}
 S.\,V.~Gryshchuk, 
{ArXiv preprint} /arXiv:1601.01626v1 [math.AP]/, 12~pages (2016)
    (on-line version:  http://arxiv.org/pdf/1601.01626.pdf).

\bibitem{Diaz-46}
J.\,B.~Diaz, 
American Journal of Mathematics {\bf 68} (4), 611--659 (1946).

\bibitem{Gilb-Hile-74}
R.\,P.~Gilbert, G.\,N.~Hile., 
Trans. Amer. Math. Soc. {\bf 195}, l--29  (1974).

\bibitem{Gilb-Wendl-Proc-75}
Robert P. Gilbert and Wolfgang L.
Wendland, 
 Proceedings of the Royal Society of
Edinburgh, Section A: Mathematics {\bf 73}, 317--331 (1975).

\bibitem{Edenhofer}
J.\,A~Edenhofer,  "Solution of the Biharmonic Dirichlet Problem by
means of Hypercomplex Analytic Functions", in {\it  Functional
Theoretic Methods for Partial Differential Equations (Proc. of
Intern. Symposium Held at Darmstand, Germany, April 12~-- 15,
1976), Ser. Lecture Notes in Mathematics {\bf 561}}, Ed. by
V.\,E.~Meister, W.\,L.~Wendland, N.~Weck, (Springer-Verlag, Berlin
etc., 1976), pp.~192--202.

\bibitem{Hile-79}
G.\,N.~Hile, 
Journal of Differential Equations  {\bf 32} (3), 369--387 (1979).

\bibitem{Sold-ell-vys-por}
A.\,P.~Soldatov, 
Differ. Uravn. {\bf 25} (1), 136~--~144 (1989) [in Russian];
English transl. in  Differ. Equations {\bf 25} (1) 109--115
(1989).

\bibitem{Sold-Izv-06}
A.\,P.~Soldatov, 
Izv. Ross. Akad. Nauk, Ser. Mat. {\bf 70} (6), 161--192 (2006) [in
Russian]; English transl. in  Izv. Math. {\bf 70} (6), 1233--1264
(2006).

\bibitem{Boj-zad-Neymn-Big-Equ} Tsoi Sun Bon, 
Differentsial'nye Uravneniya {\bf 27} (1), 169--172 (1991) [in
Russian].

\end{thebibliography}
\end{document}